\documentclass[a4paper]{amsart}

\usepackage{hodge}

\makeatletter
\@namedef{subjclassname@2020}{%
  \textup{2020} Mathematics Subject Classification}
\makeatother

\begin{document}

\title[$p$-adic families of modular forms for Hodge type Shimura varieties]{$p$-adic families of modular forms for Hodge type Shimura varieties with non-empty ordinary locus}

\author{Riccardo Brasca}
\email{\href{mailto:riccardo.brasca@imj-prg.fr}{riccardo.brasca@imj-prg.fr}}
\urladdr{\url{http://www.imj-prg.fr/~riccardo.brasca/}}
\address{Institut de Math\'ematiques de Jussieu-Paris Rive Gauche\\
Université de Paris\\
Paris\\
France}

\subjclass[2020]{Primary: 11F33; Secondary: 11F55}
\keywords{$p$-adic families of modular forms, Hodge type Shimura varieties}

\begin{abstract}
We generalize some of the results of \cite{AIP} and \cite{pel} to Hodge type Shimura varieties having non-empty ordinary locus. For any $p$-adic weight $\kappa$, we give a geometric definition of the space of overconvergent modular forms of weight $\kappa$ in terms of sections of a sheaf. We show that our sheaves live in analytic families, interpolating the classical sheaves for integral weights. We  define an action of the Hecke algebra, including a completely continuous operator at $p$. In some simple cases,  we also build the eigenvariety.
\end{abstract}

\date{\today}
\maketitle

\section*{Introduction}
Let $p$ be a fixed prime number. The study of congruences modulo $p$ between modular forms has a long story, going back to Serre and Katz in the seventies and then to the work of Hida in the eighties about ordinary forms. The theory for $\GL_2/\Q$ culminated with the construction, by Coleman and Mazur, of the \emph{eigencurve}, that parameterizes system of eigenvalues for finite slope overconvergent modular forms. This subject has become more and more important in number theory. It was already a fundamental tool in the proof of Fermat's last theorem, and the spectacular recent results about potential modularity of Galois representations rely, among other things, on the theory of $p$-adic families of automorphic forms.

A good theory of $p$-adic families of Galois representations exists for many algebraic groups different from $\GL_2/\Q$, so the Langlands program philosophy suggests that there should be a good theory of $p$-adic families of automorphic forms for a large class of reductive groups. The approach of Coleman and Mazur, based on $q$-expansion, seems difficult to generalize. On the other hand, using the theory of modular symbols of Ash and Stevens in \cite{AshSt}, Urban in \cite{urban} and Hansen in \cite{hansen}, were able to construct families for any reductive group with discrete series. Another possible approach, still cohomological, is the use of completed cohomology as done by Emerton in \cite{emerton}. In all these constructions, the theory focuses on systems of eigenvalues appearing in the space of overconvergent modular forms (or in the space of some cohomology classes) rather than on modular forms themselves. One of the reasons for this is the lack of the notion of families of overconvergent modular forms (while the notion of family of systems of eigenvalues is easier to define). In the present paper  we define the notion of families of  forms, rather then just the notion of system of eigenvalues. The starting point is the geometric approach introduced by Andreatta, Iovita, Pilloni, and Stevens in a series of papers. The basic idea is quite simple: analytically interpolate the sheaves $\omega^k$, where $k$ is an integral weight, defining, over some strict neighborhood of the ordinary locus, the sheaves $\omega^\kappa$ for any $p$-adic weight $\kappa$. This has been done for Hilbert modular varieties in \cite{over, vincent, AIPhilb}, for Siegel varieties in \cite{AIP} and for PEL-type Shimura varieties with non-empty ordinary locus in \cite{pel}. (All these constructions are for cuspidal forms, see \cite{BrasRos} for what can be done without the cuspidality assumption.) The case of PEL type Shimura varieties without ordinary locus has been considered in \cite{mu_eigen, EischenMantovan} for ordinary forms and in \cite{hernandez} for overconvergent forms. Our goal in this paper is to generalize the construction to the case of Shimura varieties of Hodge type with ordinary locus.

Let us be more precise about our results. We assume that $p \neq 2$. Let $(G,X)$ be a Shimura datum of Hodge type, with associated Shimura variety $S$. By definition this means that there is a closed immersion $S \hookrightarrow \widetilde S$, where $\widetilde S$ is a Siegel variety. We work with adic spaces, but we also need a good theory of integral model of our varieties and sheaves, to use \cite{kisin} and \cite{lovering}, so we make some technical assumptions about $G$ and the level structure, that we will not recall in this introduction, see Subsection \ref{subsec: hodge type} for details. The integral (dominant) weights for $S$ parameterize irreducible algebraic representations of the Levi of a fixed parabolic subgroup of $G$. These weights can be interpolated by the weight space $\mathcal{W}$, that parameterizes $p$-adic weights. For any integral weight $k$, the general theory of Shimura varieties provides us with an automorphic sheaf $\omega^k$ on $S$. Let now $\widetilde S(v)$ be the strict neighborhood of the ordinary locus of $\widetilde S$ defined by the condition that the Hasse invariant has valuation smaller or equal than $v$ and let us denote by $S(v)$ its pullback to $S$. \emph{We assume $S(v)$ to be non-empty}. Note that we did not specify the level of the Shimura variety: as usual outside $p$ it only needs to be small enough, but at $p$ it is Iwahoric, see Subsection \ref{subsec: modifications}. Our main result is the following
\begin{teono}
Let $\kappa \in \mathcal{W}$ be a $p$-adic weight. Then there is a $v > 0$, depending on $\kappa$, and a sheaf $\omega^\kappa$ on $S(v)$ such that the following holds. The sheaves $\omega^\kappa$ live in analytic families and interpolate the sheaves $\omega^k$, for $k$ an integral weight. Moreover, we have an action of the unramified Hecke algebra for $G$ on the space of global sections of $\omega^\kappa$ and also an action of a completely continuous operator at $p$.
\end{teono}
In the Siegel case, the strategy to prove the theorem is the following. Let $\widetilde A \to \widetilde S$ be the universal abelian variety over $\widetilde S$ and let $\omega$ be its conormal sheaf. Using the Hodge-Tate morphism (see \cite[Subsection 3.2]{AIP}), one constructs a \emph{new integral structure} $\mathcal{F} \subseteq \omega$. One then proves that $\mathcal{F}$ is a locally free sheaf and that $\mathcal{F}/p$ is equipped with a local basis of global sections, called the \emph{marked sections}. Essentially they are the image, via the Hodge-Tate map, of the universal sections of the canonical subgroup, that exist by definition thanks to the level at $p$ of $\widetilde S$. One then builds a torsor, for the Iwahori subgroup of $\GL_g$, considering certain basis of $\mathcal{F}$ that are compatible with the marked sections. The required sheaves are obtained by considering functions on this torsor that are homogeneous of the required weight for the action of the Iwahori subgroup. Here we are simplifying the situation (see \cite[Section 5]{AIP}), but we want to stress that, once one has the sheaf $\mathcal{F}$ and the marked sections, the construction, even if technically difficult, is rather formal. In the PEL case the situation is the same: the canonical subgroup and the Hodge-Tate map respect the PEL structure, for example the canonical subgroup is automatically stable under the action of the totally real field, and the Hodge-Tate map is equivariant. This allows one to take into account the extra structure everywhere and to build a torsor for the Iwahori of the relevant group. The definition of the sheaves is then similar to the Siegel case.

The situation is different for Shimura varieties of Hodge type. First of all, the embedding $S \hookrightarrow \widetilde S$ does not give any canonical extra structure to work with. The solution to this problem, that goes back to Deligne, is to introduce a non canonical extra structure, as follows. Let $V$ be the symplectic space associated to $\widetilde S$. Then, there exists $W \subseteq V^\otimes \colonequals \bigoplus_{a, b \in \N} V^{\otimes a} \otimes (V^\vee)^{\otimes b}$ such that $G \subseteq \GSp(V)$ is the subgroup that stabilizes $W$. There is a natural filtration on $V$ and we will write $\Sigma$ for $\Fil^1(V)$. It is a representation of $\widetilde M$, the Levi factor of the Siegel parabolic of $\GSp(V)$. The automorphic sheaf $\Sigma_{\dR}$ on $S$ associated to $\Sigma$ is by definition the pullback to $S$ of $\omega$. There is a filtration also on $W$, and we let $\Theta$ be the associated graded object. This is a representation of $M$, the Levi of the Siegel parabolic of $G$, and so it has an associated sheaf $\Theta_{\dR}$, that does not come from $\widetilde S$. All this objects are equipped with the integral structures, denoted with a $+$ as superscript. In particular we have $\Sigma_{\dR}^+ \subseteq \Sigma_{\dR}$ and we also have the modification $\Sigma_{\dR}^\sharp$ of $\Sigma_{\dR}^+$, corresponding to the pullback to $S$ of $\mathcal{F}$. We have that $\Sigma_{\dR}^\sharp/p$ is equipped with a set of marked sections. The main technical difficulty of our work is the definition of a modification $\Theta_{\dR}^\sharp$ of $\Theta_{\dR}^+$ and a set of marked sections of $\Theta_{\dR}^\sharp/p$. By construction we have $\Theta_{\dR} \subseteq \Sigma_{\dR}^\otimes$ and we define
\[
\Theta^\sharp_{\dR} \colonequals \Theta_{\dR} \bigcap \left( \Sigma_{\dR}^\sharp \right)^\otimes.
\]
Note that, even if $\Sigma_{\dR}^\sharp$ is locally free, it is not clear whether the same holds for $\Theta_{\dR}^\sharp$. To prove this, and to get the marked sections, we consider the Hodge-Tate map
\[
\HT \colon C^\otimes \to \left( \Sigma_{\dR}^\sharp \right)^\otimes /p,
\]
where $C$ is the canonical subgroup. (To be more precise, we will need to consider higher level canonical subgroups, but let us ignore this issue.) Thanks to the level structure at $p$, we have that $C$ is canonically isomorphic to $\Sigma^+/p$. In particular we can view $W^+/p$ as a subspace of $C^\otimes$. It is equipped with a canonical basis. We then have the following result, see Proposition \ref{prop: sharp works}.
\begin{keyprop}
The image via $\HT$ of $W^+/p \subseteq C^\otimes$ is included in $\Theta_{\dR}^\sharp/p$, and the image of the canonical basis of $W^+/p$ gives a set of marked sections of $\Theta_{\dR}^\sharp/p$. Moreover, $\Theta_{\dR}^\sharp$ is a locally free sheaf.
\end{keyprop}
We want to stress that the definition of the subspace $W^+/p \subseteq C^\otimes$, that corresponds to the extra structure that exists on the canonical subgroup, is rather {\it{ad hoc}} and so it is not clear a priori that the Hodge-Tate maps should respect it. The proof of the proposition consists in a detailed study of the Hodge-Tate map, using the fact that the result is true on the ordinary locus, where $\Sigma_{\dR}^+ = \Sigma_{\dR}^\sharp$. This proposition is related to a result of Caraiani and Scholze in \cite{car-scho} concerning the Hodge-Tate filtration of $A$, see Remark \ref{rmk: caraiani scholze}. Once we have $\Sigma_{\dR}^\sharp$ and the set of marked sections, the definition of the sheaves works without much difficulties.

Here is a detailed explanation of what we do in each section of the paper. In Section \ref{sec: alg theory} we introduce the Shimura varieties we will work with, and their integral models. We also fix the various representation theory data we will need and we recall how automorphic sheaves are associated to algebraic representation in our setting. We then introduce classical modular forms for $G$. Section \ref{sec: p-adic} is the core of our work. We first of all recall the theory of the canonical subgroup and we introduce the tower of Shimura varieties we need. These are called $\widetilde{\Ig}$ and $\Ig$: the former comes from $\widetilde S$ and essentially parameterizes trivializations of the canonical subgroup, ignoring the extra datum coming from $W$, while the latter takes into account this structure. We then prove our key proposition and we explain why our construction does not depend on the choice of $W$. Finally we introduce $p$-adic families of modular forms. In Section \ref{sec: Hecke} we introduce Hecke operators. The action of the unramified Hecke algebra poses no problem. The $\U$ operator is more complicated: we define it as the trace of the Frobenius morphism, that is defined as usual taking the quotient by the canonical subgroup. Here we have to check that if $x \in S$ corresponds to an abelian variety $A$ with canonical subgroup $C$, then the point of $\widetilde S$ corresponding to $A/C$ also lies in $S$ (something very easy in the PEL case). To prove this we use a result of Noot in \cite{noot} that describes $S(0)$ inside $\widetilde S(0)$ using Serre-Tate coordinates: we have that $S(0)$ corresponds to a subtorus of the completed local ring of $\widetilde S(0)$ and in particular it is stable under the quotient by the kernel of the Frobenius. We finally show that the $\U$ operator is completely continuous. This allows us to build the eigenvariety in the proper case. In Section \ref{sec: orthogonal} we give an example of our construction, considering the case of orthogonal Shimura varieties associated to a spinor group. We also explain how our construction gives the eigenvariety for the modular forms considered in Borcherds' theory.

Let us finish this introduction explaining what we do \emph{not} do in this paper. We construct the eigenvariety only on some specific cases, see Subsections \ref{subsec: eigenvar proper} and \ref{subsec: orth form}. To obtain them in general, first of all one should extend our construction to a fixed toroidal compactification of $S$. These compactifications, even at integral level, are constructed by Madapusi Pera in \cite{madapusi}: since they carry a universal semi-abelian scheme, we think that our constructions extend, but we will not write the details. In any case to obtain the eigenvariety one would need to have a vanishing result for the higher direct images from the toroidal to the minimal compactification of the sheaf of cuspidal forms, as in \cite[Subsection 8.2]{AIP}. This has been proved for PEL type Shimura varieties by Lan in \cite{lan_ram} and used in \cite{pel}, but it is not known for general Hodge type Shimura varieties. We do not have any classicity result: the techniques used in \cite{class} for the PEL case, that rely on the study of the geometry of $S$, seem difficult to adapt to our situation. Finally, we treat in this paper the case of Shimura varieties of Hodge type but we think that our techniques should work also for Shimura varieties of abelian type. 

Recently Boxer and Pilloni have announced a very general construction of eigenvarieties for Shimura varieties of abelian type without any assumption on the non-emptiness of the ordinary locus. We still believe that, in our more restrictive situation, our presentation is so elementary and straightforward that it might be interesting to have it recorded.

\subsection*{Acknowledgments} We would like to thank Fabrizio Andreatta for suggesting the problem and for several useful discussions.

\subsection*{Notations and conventions} Let $p > 2$ be a prime, fixed from now on. We fix an algebraic closure $\overline \Q_p$ of $\Q_p$ and an isomorphism $\C \stackrel{\sim}{\longrightarrow} \overline \Q_p$. We will write $\A_f$ for the ring of finite adeles of $\Q$ and $\A_f^p$ for the finite adeles away from $p$. The profinite completion of $\Z$ will be denoted by $\widehat \Z$. We will write $M^\otimes$ for the direct sum of all objects obtained from $M$ using tensor products and duals (in the appropriate category), so $M^\otimes = \bigoplus_{a, b \in \N} M^{\otimes a} \otimes (M^\vee)^{\otimes b}$. Usually, base-change will be denoted with a subscript.

Through the paper, objects relative to the Siegel variety will be denoted using a tilde, for example the group of symplectic similitudes will be denoted by $\widetilde G$, and the corresponding objects for the Hodge type Shimura variety will be denoted by the same symbol, without the tilde, so the relevant reductive group will be denoted by $G$. We will explain our constructions in the Hodge type case, but we will often use the corresponding constructions in the Siegel case just adding the tilde.

\section{Algebraic theory} \label{sec: alg theory}
\subsection{Shimura varieties of Hodge type} \label{subsec: hodge type} If $g \geq 1$ is an integer, we consider the group $\widetilde G =\GSp_{2g} / \Q$ of symplectic similitudes of $V = \Q^{2g}$ equipped with the pairing $\Psi$ induced by the matrix $\begin{pmatrix}  & s \\ -s & \end{pmatrix}$, where $s$ is the $g\times g$ anti-diagonal matrix with non-zero entries equal to $1$. Let $S^\pm$ be the usual union of the Siegel upper-half spaces. Then $(\widetilde G, S^\pm)$ is a Shimura datum. If $\widetilde K \subseteq \widetilde G(\A_f)$ is a neat compact open subgroup, then we obtain the Shimura variety $\widetilde S_{\widetilde K}$ over $\Q_p$. We have that $\widetilde S_{\widetilde K}$ is a fine moduli space of $g$-dimensional principally polarized abelian varieties, with $\widetilde K$-level structure.

Let $\Lambda \subseteq V$ be a $\Z$-lattice such that $\Lambda_{\widehat \Z}$ is $\widetilde K$-stable. We suppose moreover that $\widetilde K$ is of the form $\widetilde K = \widetilde K^p \times \widetilde K_p$, where $\widetilde K^p \subseteq \widetilde G(\A_f^p)$ is a compact open and $\widetilde K_p \subseteq \widetilde G(\Q_p)$ is a maximal compact open subgroup. If $\widetilde K^p$ is small enough, then there is a canonical integral model $\widetilde{\mathcal{S}}_{\widetilde K}$ of $\widetilde S_{\widetilde K}$ over $\Z_p$ that solves the same moduli problem. We have that $\widetilde{\mathcal{S}}_{\widetilde K}$ is smooth if and only if $\widetilde K^p$ is hyperspecial, which is equivalent to the restriction of $\Psi$ to $\Lambda_{\Z_p}$ being a perfect pairing.

Let $(G,X)$ be a Shimura datum of \emph{Hodge type}. This means that there exist an embedding $(G,X) \hookrightarrow (\widetilde G, S^\pm)$ of $(G,X)$ into a Siegel datum as above. We write $E$ for the completion of the reflex field of $(G,X)$ with respect to our fixed isomorphism $\C \stackrel{\sim}{\longrightarrow} \overline \Q_p$. It is a finite extension of $\Q_p$. Let $K \subseteq G(\A_f)$ be a neat compact open subgroup. We hence have the Shimura variety $S_K$ defined over $E$. By the results of \cite{deligne}, there is a compact open $\widetilde K \subseteq \widetilde G(\A_f)$ such that $K = \widetilde K \cap \widetilde G(\A_f)$ and there is a closed immersion
\[
S_K \hookrightarrow \widetilde S_{\widetilde K} \otimes E.
\]
We suppose from now on that the level $K$ is of the form $K = K^p \times K_p$, where $K^p \subseteq G(\A_f^p)$ is a sufficiently small compact open and $K_p \subseteq G(\Q_p)$ is \emph{hyperspecial}. This means that $K_p = \mathcal{G}(\Z_p)$, where $\mathcal{G}$ is a reductive $\Z_p$-model of $G_{\Q_p}$. There is a $\Z_p$-lattice $V^+ \subseteq V_{\Q_p}$ such that the natural inclusion $G \hookrightarrow \widetilde G \subseteq \GL(V)$ extends to an embedding $\mathcal{G} \subseteq \GL(V^+)$. We may also assume that the pairing $\Psi$ extends to a perfect pairing
\[
\Psi \colon V^+ \times V^+ \to \Z_p.
\]
In particular we get an embedding $\mathcal{G} \subseteq \GSp(V^+)$. We fix once and for all such a lattice and such an embedding. This choice gives, as above, the integral model $\widetilde{\mathcal{S}}_{\widetilde K}$ over $\Z_p$. We define $\mathcal{S}_K$ to be the normalization of $S_K$ inside $\widetilde{\mathcal{S}}_{\widetilde K} \otimes \mathcal{O}_E$. By the main result of \cite{kisin}, we have that $\mathcal{S}_K$ is smooth and it is a \emph{canonical integral model} of $S_K$ (in the sense that is satisfies Milne's extension property). We write $\widetilde{\mathcal{A}} \to \widetilde{\mathcal{S}}_K$ for the universal abelian scheme. Its generic fiber will be denoted $A \to \widetilde S_K$. The pullback of $\widetilde{\mathcal{A}}$ and $\widetilde A$ to $\mathcal{S}_K$ and $S_K$ will be denoted by $\mathcal{A}$ and $A$.
\begin{ass} \label{ass: ord dense}
We assume that the ordinary locus of the reduction of $\mathcal{S}_K$ is dense.
\end{ass}
We consider the level $K$ fixed from now on, and we erase it from the notation, writing $\mathcal{S}$ for our Shimura variety. We write $S^{\ad} \to \Spa(\Q_p, \Z_p)$ for the adic space associated to $(S,\mathcal{S})$ (so $\mathcal{O}_{S^{\ad}}^+ = \mathcal{O}_{\mathcal{S}}$), and similarly for $\widetilde S^{\ad}$.

In the following, we will need to work over $K$, a finite extension of $\Q_p$ (no confusion should arise since we will not need the level anymore). We will assume without any comment that $K$ is large enough and we base change all our objects to $K$, without any explicit notation. Indeed, $K$ will not play any significant role in our theory.
\subsection{Parabolic subgroups} Let $\widetilde Q$ be the Siegel parabolic of $\widetilde G$ consisting of matrices whose lower left $g\times g$ block is $0$ and we write $\widetilde M$ for its Levi subgroup. We fix once and for all the embedding $\GL_g \hookrightarrow \widetilde G$ given by $a \mapsto \begin{pmatrix} a & 0 \\ 0 & ^ta^{-1} \end{pmatrix}$. Writing $\widetilde G^\circ$ for $\mathrm{Sp}_{2g}/\Q$, and similarly for the other groups, the previous embedding gives an isomorphism $\GL_g \stackrel{\sim}{\longrightarrow} \widetilde M^\circ$. We can attach to any representation $\rho$ of $\widetilde Q$ an automorphic sheaf $\widetilde \E(\rho)$ on $\widetilde S$. The parabolic $\widetilde Q$ has a canonical extension over $\Z$ to $\widetilde{\mathcal{Q}} \subseteq \GSp_{2g}/\Z$, and similarly for $\widetilde{\mathcal{M}}$ etc. If $\rho$ extends to a representation of $\widetilde{\mathcal{Q}}$ we have a canonical integral model $\widetilde \E(\rho)^+$ of $\widetilde \E(\rho)$. We consider $\widetilde \E(\rho)$, resp. $\widetilde \E(\rho)^+$, as a sheaf of $\mathcal{O}_{\widetilde S^{\ad}}$-modules, resp. $\mathcal{O}_{\widetilde S^{\ad}}^+$-modules, on $\widetilde S^{\ad}$. Note that $\widetilde{\mathcal{E}}(\rho)$ comes equipped with a filtration.

We set $Q \colonequals \widetilde Q \cap G$, $M \colonequals \widetilde M \cap G$ and similarly for $Q^\circ$ and the other groups. We have that $Q$ extends to a parabolic subgroup $\mathcal{Q} \subseteq \mathcal{G}$ and we use the obvious notation for its Levi and derived subgroup. As above, if $\rho$ is a representation of $Q$, we have an automorphic sheaf $\E(\rho)$ on $S$. If $\rho$ is the restriction of a $\widetilde Q$-representation, then $\E(\rho)$ is the pullback of $\widetilde{\mathcal{E}}(\rho)$ via $S \hookrightarrow \widetilde S$. As above, we consider $\E(\rho)$ as a sheaf of $\mathcal{O}_{S^{\ad}}$-modules on $S^{\ad}$. Again $\mathcal{E}(\rho)$ is naturally equipped with a filtration. If $\rho$ extends to a $\mathcal{Q}$-representation, then, by \cite{lovering}, we have a canonical integral model $\mathcal{E}(\rho)^+$ of $\mathcal{E}(\rho)$.
\subsection{Borel and unipotent subgroups}
Recall that we are working over $K$, a sufficiently big finite extension of $\Q_p$, in particular we can assume that $G$ is split over $K$. We let $B$ be a fixed Borel subgroup of $M^\circ$ and we write $T$ and $U$ for the maximal torus and the unipotent part of $B$, so $B = TU$. We denote by $\widetilde B $ a Borel of $\widetilde M^\circ \cong \GL_g$ such that $B = \widetilde B \cap G$. We write $\widetilde T$ and $\widetilde U$ for the torus and unipotent part of $\widetilde B$.

We say that a weight of $T$ is dominant if it is dominant with respect to $B$ and we denote by $X^\ast(T)^+$ the cone of dominant weights. We write $\kappa' = -w_0 \kappa$, where $w_0$ is the longest element of the Weyl group of $M^\circ$. The involution $\kappa \mapsto \kappa'$ of $X^\ast(T)$ respects the cone $X^\ast(T)^+$. If $\kappa \in X^\ast(T)^+$, we define the space
\[
V_{\kappa} \colonequals \{f \colon M^\circ \to \mathbf{A}^1 \mbox{ such that } f(gb) = \kappa(b)f(g) \mbox{ for all } (g,b) \in M^\circ \times B \},
\]
where by definition $\kappa$ is trivial on $U$. We let $M^\circ$ act on $V_{\kappa}$ via $(g \star f)(x) = f(g^{-1}x)$. Then $V_{\kappa}$ is a finite dimensional $K$-vector space and it is the irreducible $M^\circ$-representation of highest weight $\kappa' $.
\subsection{Betti and de Rham realizations} By general representation theory, there exists $W \subseteq V^\otimes_{\Q_p}$ such that
\[
G = \{g \in \widetilde G \mbox{ such that } g(W) \subseteq W \}.
\]
More precisely, by \cite[Proposition 1.3.2]{kisin}, there exist a $\Z_p$-lattice $W^+ \subseteq (V^+)^\otimes$ that is a direct summand and such that
\[
\mathcal{G} = \{g \in \GSp(V^+) \mbox{ such that } g(W^+) \subseteq W^+ \}.
\]
We fix such $W$ and $W^+$ once and for all. By \cite{kisin}, there exists a $p$-adic Hodge cocharacter $w\colon \mathbf{G}_{\mathrm{m}} \to \mathcal{G}$ defined over  $\Z_p$ that factors through $\mathcal{Q}$ and defines $\mathcal{Q}$: it is the parabolic subgroup of $\mathcal{G}$ such that its Lie algebra $\mathrm{Lie} \, \mathcal{Q}$ is the $\Z_p$-submodule of the Lie algebra $\mathrm{Lie} \, \mathcal{G}$  of $\mathcal{G}$ on which the adjoint action of $\mathbf{G}_{\mathrm{m}}$ given by $w$ has non-negative weights.  The cocharacter $w$ of $\mathcal{G}$, and of $\widetilde{\mathcal{G}}$ via the inclusion $\mathcal{G}\subset \widetilde{\mathcal{G}}$, define a grading on $V^+$ and $W^+$. We will be interested in the associated graded objects
\[
\Theta \colonequals \Gr^\bullet(W) \mbox{ and } \Theta^+ \colonequals \Gr^\bullet(W^+).
\]
Since the filtrations on $V_{\Q_p}$ and $V^+$ have only two nontrivial steps, that are moreover one dual to the other, we have 
\[
\Theta \subseteq \Sigma^\otimes \mbox{ and } \Theta^+ \subseteq (\Sigma^+)^\otimes.
\]
As the gradings are defined by the cocharacter $w$,  we see that the inclusions above are strict, i.e., the grading on $\Theta$ and $\Theta^+$ are those induced by the grading on $\Sigma^\otimes$ and $(\Sigma^+)^\otimes$ respectively. We fix once and for all an isomorphism $\Sigma^+ \cong \Z_p^g$, so there is an integer $t \in \N$ such that
\[
\Theta^+ \cong \Z_p^t \subseteq \bigoplus_{a,b \in \N} (\Z_p^g){\otimes a} \otimes (\Z_p^g)^{\otimes b}
\]
as a direct summand using our conventions. Since the parabolic subgroup $\mathcal{Q}$ is normalized by the cocharacter $w$, its Levi subgroup $\widetilde{\mathcal{M}}$ acts on $\Sigma^+$ and
\[
\mathcal{M}^\circ = \{g \in \widetilde{\mathcal{M}}^\circ \mbox{ such that } g(\Theta^+) \subseteq \Theta^+ \},
\]
where $\mathcal{M}^\circ$ and $\widetilde{\mathcal{M}}^\circ$ have the obvious definitions. We now move on to de Rham realizations. We have $\mathcal{E}(V^+)^+ = \Homol^1_{\dR}(\mathcal{A})$ and we define
\[
\Sigma_{\dR}^+ \colonequals \Fil^1(\mathcal{E}(V^+)^+) = \omega^+,
\]
where $\omega^+$ is the conormal sheaf of $\mathcal{A}$. It is a locally free sheaf of $\mathcal{O}_{S^{\ad}}^+$-modules. We also define
\[
\Theta^+_{\dR} \colonequals \Gr^\bullet(\mathcal{E}(W^+)^+).
\]
It is a locally free sheaf of $\mathcal{O}_{S^{\ad}}^+$-modules of rank $t$ and, since $\mathcal{E}(\cdot)^+$ is an exact tensor functor, we have
\[
\Theta^+_{\dR} \subseteq (\Sigma^+_{\dR})^\otimes
\]
as a direct summand. We extend the notation to objects in characteristic $0$ in the obvious way.
\subsection{Classical modular forms} Note that if $\imath \colon \Sigma \otimes \mathcal{O}_{X^{\ad}} \stackrel{\sim}{\longrightarrow} \Sigma_{\dR}$ is an isomorphism then, using the polarization on $A$, we obtain an isomorphism $\Sigma^\otimes \otimes \mathcal{O}_{X^{\ad}} \stackrel{\sim}{\longrightarrow} \Sigma_{\dR}^\otimes$. In particular, we consider the space
\[
\mathcal{T} \colonequals \{ \imath \in \underline{\mathrm{Isom}}(\Sigma \otimes \mathcal{O}_{X^{\ad}}, \Sigma_{\dR}) \mbox{ such that } \imath(\Theta \otimes \mathcal{O}_{X^{\ad}}) \subseteq \Theta_{\dR}  \}.
\]
We define a left action $M^\circ \times \mathcal{T} \to \mathcal{T}$ by sending an isomorphism $\imath$ to $\imath \circ g^{-1}$ for all $g \in M^\circ$. By definition $f \colon \mathcal{T} \to S^{\ad}$ is an $M^\circ$-torsor. We also have the $\widetilde M^\circ = \GL_g$-torsor $\widetilde f \colon \widetilde{\mathcal{T}} \to \widetilde S^{\ad}$ defined as $\mathcal{T}$ but without the condition on $\Theta$.
\begin{defi}
Let $\kappa \in X^\ast(T)^+$ be a dominant weight. We define the sheaf $\omega^\kappa$ on $S^{\ad}$ by
\[
\omega^\kappa \colonequals f_\ast \mathcal{O_{\mathcal{T}}}[\kappa],
\]
meaning the sheaf of functions that are $\kappa$-equivariant for the action of $B$.
\end{defi}
Locally for the étale topology on $S^{\ad}$, we have an isomorphism
\[
\omega^\kappa \cong V_{\kappa}.
\]

\section{\texorpdfstring{$p$-adic theory}{p-adic theory}} \label{sec: p-adic}
In this section we want to put the sheaves $\omega^\kappa$ in $p$-adic families.
\subsection{Canonical subgroups and modifications} \label{subsec: modifications} Over $\mathcal{S}_{\F_p}$ we have the \emph{Hasse invariant}, that is a modular form of weight $p-1$. There is a continuous function
\[
\Hdg \colon S^{\ad} \to [0,1]
\]
defined taking the truncated valuation of any lift of the Hasse invariant. We define
\[
S(v)^{\ad} \colonequals \Hdg^{-1}([0,v]).
\]
In particular $S(0)^{\ad}$, that it is dense by Assumption~\ref{ass: ord dense}, is the inverse image of the ordinary locus under the specialization map and $S(v)^{\ad}$ is a strict neighborhood of $S(0)^{\ad}$. Let $n \geq 1$ be an integer. By \cite{fargues_can}, if $v < \frac{1}{p^{n+1}}$, the universal $p$-divisible group $A[p^\infty]$ over $S(v)^{\ad}$ as a canonical subgroup $H_n \subseteq A[p^n]$ of level $n$. We write $\widetilde\Ig_n(v) \to S(v)^{\ad}$ for the étale cover classifying trivializations $(\Z/p^n\Z)^g \cong H_n^{\D}$ of the Cartier dual of $H_n$. Over $\widetilde\Ig_n(v)$ we have
\begin{itemize}
 \item the $\mathcal{O}_{\widetilde\Ig_n(v)}$-module $\omega$ defined as the conormal sheaf of the universal abelian variety $A \to \widetilde\Ig_n(v)$;
 \item the $\mathcal{O}_{\widetilde\Ig_n(v)}^+$-module $\omega^+$, that is an integral structure of $\omega$, defined, over an open affinoid $\Spa(B,B^+) \subseteq \widetilde\Ig_n(v)$, as the conormal sheaf of $A^+$, where $A^+$ is an extension to $B^+$ of the pullback of $A$;
 \item the $\mathcal{O}_{\widetilde\Ig_n(v)}^+$-module $\omega_{H_n}^+$ defined similarly to $\omega^+$.
\end{itemize}
The \emph{Hodge-Tate map} is a morphism $\HT \colon \T_p(A) \to \omega^+$. We also have $\HT_{H_n} \colon H_n^{\D} \to \omega_{H_n}^+$. We let $\omega^\sharp \subseteq \omega^+$ be the inverse image via $\omega^+ \to \omega_{H_n}^+$ of $\HT_{H_n}(H_n^{\D} \otimes \mathcal{O}_{\widetilde\Ig_n(v)}^+)$. By \cite{AIP}, we have that $\omega^\sharp$ does not depend on $n$ and it is a locally free sheaf of rank $g$ such that $\omega^\sharp \otimes \mathcal{O}_{\widetilde\Ig_n(v)} \cong \omega$ and moreover, if $w$ is a rational number such that $0 < w \leq n - v\frac{p^n}{p-1}$, then $\HT$ induces an isomorphism of $\mathcal{O}_{\widetilde\Ig_n(v)}^+/p^w$-modules
\[
\HT_w \colon H_n^{\D} \otimes \mathcal{O}_{\widetilde\Ig_n(v)}^+/p^w \stackrel{\sim}{\longrightarrow} \omega^\sharp/p^w.
\]
We call $\omega^\sharp$ a \emph{modification} of $\omega^+$ and we set
\[
\Sigma^\sharp_{\dR} \colonequals \omega^\sharp.
\]
Using the image via $\HT_w$ of the canonical basis of $H_n^{\D}$ that exists over $\widetilde{\Ig}_n(v)$, we see that $\Sigma^\sharp_{\dR}/p^w$ is equipped with a set of \emph{marked sections}, namely sections of $\Sigma^\sharp_{\dR}/p^w$ that are locally a basis. Our goal in this section is to define a modification $\Theta^\sharp_{\dR}$ of $\Theta^+_{\dR}$ and to show that $\Theta^\sharp_{\dR}/p^w$ is equipped with a set of marked sections.

We write $\Sigma_{\et,n}^+$ for $H_n^{\D}$, so $\HT_w$ is a map $\Sigma_{\et,n}^+ \to \Sigma_{\dR}^\sharp/p^w$ that induces an isomorphism
\[
\HT_w \colon \Sigma_{\et,n}^+ \otimes \mathcal{O}_{\widetilde\Ig_n(v)}^+/p^w \stackrel{\sim}{\longrightarrow} \Sigma_{\dR}^\sharp/p^w.
\]
The dual of $\HT_w^{-1} \colon \Sigma_{\dR}^\sharp/p^w \to \Sigma_{\et,n}^+ \otimes \mathcal{O}_{\widetilde\Ig_n(v)}^+/p^w$ induces a morphism, denoted again $\HT_w$,
\[
\HT_w \colon \left(\Sigma^+_{\et,n}\right)^\otimes \to \left(\Sigma^+_{\dR}/p^w \right)^\otimes.
\]
Recall that over $\widetilde{\Ig}_n(v)$ we have the identification $\Sigma^+_{\et,n} = \Sigma^+/p^n$. We define $\Theta^+_{\et,n}$ as the submodule $\Theta^+_{\et,n} \subseteq (\Sigma^+_{\et,n})^\otimes$ making the following diagram cartesian
\[
\begin{tikzcd}
\Theta^+_{\et,n} \arrow[hook]{r} \ar{d}{\wr} & \left(\Sigma^+_{\et,n}\right)^\otimes \ar[equal]{d} \\
\Theta^+/p^n \arrow[hook]{r} & \left(\Sigma^+/p^n\right)^\otimes
\end{tikzcd}
\]
Clearly $\Theta^+_{\et,n} \subseteq (\Sigma^+_{\et,n})^\otimes$ is a direct summand and moreover we have a morphism $\HT_w \colon \Theta^+_{\et,n} \to \left(\Sigma^+_{\dR}/p^w \right)^\otimes$. By the properties of the functor $\mathcal{E}(\cdot)^+$, the image of $\Theta^+$ via $\HT$ is inside $\Theta^+_{\dR}$, giving the morphism
\[
\HT_w \colon \Theta^+_{\et,n} \to \Theta^+_{\dR}/p^w.
\]
\begin{defi}
We define
\[
\Theta^\sharp_{\dR} \colonequals \Theta_{\dR} \bigcap \left( \Sigma_{\dR}^\sharp \right)^\otimes.
\]
\end{defi}
Recall that $\Theta^+/p^n \cong (\Z/p^n\Z)^t$, so as abstract groups $(\Z/p^n\Z)^t$ and $\Theta^+_{\et,n}$ are isomorphic. We write $\Ig_n(v) \to \widetilde{\Ig}_n(v)$ for the étale cover classifying trivializations $(\Z/p^n\Z)^t \cong \Theta^+_{\et,n}$. In particular over $\Ig_n(v)$ we have the universal isomorphism $\Psi \colon (\Z/p^n\Z)^t \stackrel{\sim}{\longrightarrow} \Theta^+_{\et,n}$ and the sections $\theta_i \colonequals \HT_w(\Psi(f_i))$, where $\{f_1,\ldots,f_t\}$ is the canonical basis of $(\Z/p^n\Z)^t$. Note that these are a priori sections of $(\Sigma^+_{\dR}/p^w )^\otimes$. From now on we will work over $\Ig_n(v)$ and we are going to prove that $\Theta^\sharp_{\dR}$ is a locally free sheaf of $\mathcal{O}_{\Ig_n(v)}^+$-modules included as a locally direct summand in $(\Sigma^+_{\dR})^\otimes$ and that $\{\theta_i\}$ is a set of marked sections of $\Theta^\sharp_{\dR}/p^w$.

We work locally over a fixed affinoid $\Spa(R,R^+) \subseteq \Ig_n(v)$ such that $\Sigma^\sharp_{\dR}$ is free. We also assume that $\Theta^+_{\dR}$ is a free $R^+$-module, of rank $t$. Let $\Ha \in R^+$ by any fixed lifting of the Hasse invariant, so $\Ha$ divides $p^v$ in $R^+$.
\begin{lemma}
The natural morphism $\Theta^\sharp_{\dR} / p^w \to (\Sigma^\sharp_{\dR})^\otimes / p^w$ is injective.
\end{lemma}
\begin{proof}
Let $\theta$ an element in $\Theta^\sharp_{\dR} / p^w$ and let $\widetilde \theta \in \Theta^\sharp_{\dR}$ be a lifting of $\theta$. Let $k \in \N$ such that $p^k \widetilde \theta \in \Theta_{\dR}^+$. If $\widetilde \theta = p^w \widetilde\sigma$ for some $\widetilde\sigma \in (\Sigma^\sharp_{\dR})^\otimes$, then $p^k\widetilde \theta = p^{k+w}\widetilde\sigma$ so $\widetilde\sigma = p^{-w}\widetilde \theta \in \Theta_{\dR}$. It follows that $\widetilde \sigma \in \Theta_{\dR} \cap (\Sigma^\sharp_{\dR})^\otimes$ and then $\theta = 0$.
\end{proof}
\begin{lemma} \label{lemma: ord loc}
Over $R^+[\Ha^{-1}]$, we have $\Theta^\sharp_{\dR}[\Ha^{-1}] = \Theta_{\dR}^+[\Ha^{-1}]$.
\end{lemma}
\begin{proof}
Over $R^+[\Ha^{-1}]$ we have $\omega^+ = \omega^\sharp$, so $(\Sigma^\sharp_{\dR})^\otimes[\Ha^{-1}] = (\Sigma_{\dR}^+)^\otimes[\Ha^{-1}]$ and then $\Theta^\sharp_{\dR}[\Ha^{-1}] = \Theta_{\dR} \cap (\Sigma_{\dR}^+)^\otimes[\Ha^{-1}] \supseteq \Theta^+_{\dR}[\Ha^{-1}]$. Equality follows since $\Theta^+_{\dR} \subseteq \Theta_{\dR} \cap (\Sigma_{\dR}^+)^\otimes$ is a direct summand.
\end{proof}
Let $S\subseteq \N^2$ be a finite set such that $\Theta^+ \subseteq \Phi \colonequals \bigoplus_{(a,b) \in S} \Z_p^{\otimes a} \otimes \Z_p^{\otimes b}$. We set
\[
\Phi_{\et,n} \colonequals \bigoplus_{(a,b) \in S} \Sigma_{\et,n}^{\otimes a} \otimes \Sigma_{\et,n}^{\vee\otimes b} \mbox{ and } \Phi_{\dR} \colonequals \bigoplus_{(a,b) \in S} (\omega^+)^{\otimes a} \otimes (\omega^{+\vee})^{\otimes b}.
\]
We write $k$ for the rank of $\Phi$ as $\Z_p$-module. Recall that we have a universal basis $f_1,\ldots f_t$ of $\Theta^+_{\et,n}$. Let us complete it to a basis $f_1,\ldots f_k$ of $\Phi_{\et,n}$ (this is possible since $\Theta^+_{\et,n} \subseteq \Phi_{\et,n}$ is a direct summand). Let $\sigma_i \colonequals \HT_w(f_i) \in (\Sigma^\sharp_{\dR})^\otimes/p^w$ for $i = 1,\ldots,k$, so by definition $\sigma_i = \theta_i \in \Theta^+_{\dR}/p^w$ if $i \leq t$ and $\{\sigma_i\}_{i=1}^k$ is a basis of $\Phi_{\dR}/p^w$. For all $i$, we fix $\widetilde \sigma_i \in \Phi_{\dR}$ a lifting of $\sigma_i$.
\begin{lemma} \label{lemma: base Phi}
We have that $\{\widetilde \sigma_i\}_{i=1}^k$ is a basis of $\Phi_{\dR}$.
\end{lemma}
\begin{proof}
We already know that $\{\sigma_i\}$ is a basis of $\Phi_{\dR}/p^w$ so, by the topological Nakayama's lemma, $\{\widetilde \sigma_i\}$ is a spanning set of $\Phi_{\dR}$. If $\sum_{i=1}^k \lambda_i \widetilde \sigma_i = 0$ where $\lambda_i \in R^+$ are not all zero, then we may assume that there exist $i_0$ such that $\lambda_{i_0} \not \in p^w R^+$. Reducing modulo $p^w$ the equality $\sum_{i=1}^k \lambda_i \widetilde \sigma_i = 0$ we get a contradiction, so $\{\widetilde \sigma_i\}$ is a basis of $\Phi_{\dR}$.
\end{proof}
Let $\widetilde \theta_i \colonequals \widetilde \sigma_i $ for $i \leq t$. We assume from now on that $\widetilde \theta_i \in \Theta^{\sharp}_{\dR}[\Ha^{-1}]$ (this is possible by Lemma~\ref{lemma: ord loc}).
\begin{lemma} \label{lemma: ord loc im}
Over $R^+[\Ha^{-1}]$, the image of $\Theta^+_{\et,n}$ via $\HT_w$ lies in $\Theta^\sharp_{\dR}[\Ha^{-1}]/p^w$. Moreover, $\Theta^\sharp_{\dR}[\Ha^{-1}]$ is a free $R^+[\Ha^{-1}]$-module with basis $\{\widetilde\theta_i\}_{i=1}^t$.
\end{lemma}
\begin{proof}
Since the image of $\HT_w$ is inside $\Theta^+_{\dR}/p^w$, the first statement is clear by Lemma~\ref{lemma: ord loc}. Let $N \subseteq \Theta^\sharp_{\dR}[\Ha^{-1}]$ be the submodule generated be the $\widetilde \sigma_i$'s, for $i=1,\ldots, t$. Then $N$ and $\Theta^\sharp_{\dR}[\Ha^{-1}]$ are both free modules of rank $t$, so the quotient $\Theta^\sharp_{\dR}[\Ha^{-1}]/N$ is a torsion module that is included in $\Phi_{\dR}[\Ha^{-1}]/\langle \widetilde \sigma_i \rangle_{i=t+1}^k \cong R^+[\Ha^{-1}]^{k-t}$, so $\Theta^\sharp_{\dR}[\Ha^{-1}]/N = 0$ and $\{\widetilde \theta_i\}$ is a basis of $\Theta^\sharp_{\dR}[\Ha^{-1}]$.
\end{proof}
Let $w_0 \colonequals n - v\frac{p^n}{p-1}$, so by assumption $w \leq w_0$ and all the above results hold true for $w=w_0$. We assume from now on that $w < w_0$.
\begin{lemma}
Up to shrinking $v$, we have $\HT_w(\Theta_{\et,n}^+) \subseteq \Theta^\sharp_{\dR}/p^w$.
\end{lemma}
\begin{proof}
By Lemma~\ref{lemma: ord loc im} there is $s \in \N$ such that $\Ha^s\theta_i \in \Theta^\sharp_{\dR}/p^{w_0}$ for $i = 1,\ldots, t$. Let $\widetilde \eta_i \in \Theta^\sharp_{\dR}$ be a lifting of $\Ha^s\theta_i$. Since $\sigma_i = \theta_i$ if $i \leq t$, there is $\widetilde\gamma_i \in (\Sigma^\sharp_{\dR})^\otimes$ such that
\[
\widetilde \eta_i - \Ha^s \widetilde\sigma_i = p^{w_0} \widetilde \gamma_i
\]
for $i = 1,\ldots, t$. Let $S^+$ be the quotient of $R^+\langle X\rangle/(p^{w_0-w}-X\Ha^s)$ by its $p$-torsion. In particular $\Spa(S^+_{\Q_p},S^+)$ is an open affinoid in $\Spa(R,R^+)$ and $\Ha^s$ divides $p^{w_0-w}$ in $S^+$. It is clear that $\Sigma^+_{\dR}$, $\Sigma^\sharp_{\dR}$, and $\Theta^+_{\dR}$ commute with base change and we have a morphism $\Theta^\sharp_{\dR} \otimes_{R^+} S^+ \to \Theta^\sharp_{\dR}(S^+)$, where $\Theta^\sharp_{\dR}(S^+)$ is defined in the same way as $\Theta^\sharp_{\dR}$, but over $S^+$. Using this morphism, all the above results hold true over $\Spa(S^+_{\Q_p},S^+)$ and in practice, up to shrinking $v$, we may assume that $\Ha^s$ divides $p^{w_0-w}$ in $R^+$. In particular we can write
\[
\widetilde \eta_i = \Ha^s \left( \widetilde\sigma_i + p^w\frac{p^{w_0 - w}}{\Ha^s} \widetilde  \gamma_i \right)
\]
so there is $\widetilde \rho_i \in (\Sigma^\sharp_{\dR})^\otimes$ such that
\[
\widetilde \eta_i = \Ha^s \widetilde\rho_i \mbox{ and } \widetilde \rho_i \equiv \widetilde \sigma_i \bmod{p^w}.
\]
We claim that $\widetilde \rho_i \in \Theta^\sharp_{\dR}$. Indeed $\widetilde \rho_i \in (\Sigma^\sharp_{\dR})^\otimes$ and $\widetilde \rho_i = \Ha^{-s} \widetilde \eta_i \in \Theta_{\dR}$ (since $\Ha$ is invertible in $R$) so $\widetilde \rho_i \in \Theta_{\dR} \cap (\Sigma^\sharp_{\dR})^\otimes = \Theta_{\dR}^\sharp$. In summary, for all $i=1,\ldots, t$ we have that $\widetilde \rho_i \in \Theta^\sharp_{\dR}$ is such that its reduction modulo $p^w$ is $\widetilde \sigma_i = \HT_w(f_i)$ and in particular $\HT_w(\Theta_{\et,n}^+) \subseteq \Theta^\sharp_{\dR}/p^w$.
\end{proof}
From now on we assume that $v$ is small enough so that $\HT_w(\Theta_{\et,n}^+) \subseteq \Theta^\sharp_{\dR}/p^w$. By the previous lemma we can assume that $\widetilde \theta_i \in \Theta_{\dR}^\sharp$.
\begin{coro}
The family $\{\theta_i\}_{i=1}^t$ is a basis of $\Theta_{\dR}^\sharp/p^w$.
\end{coro}
\begin{proof}
Since $\{\theta_i\} \subseteq \{\sigma_i\}$, the $\theta_i$'s are linearly independent by Lemma~\ref{lemma: base Phi}. Let $\theta \in \Theta_{\dR}^\sharp/p^w$ with lifting $\widetilde \theta \in \Theta_{\dR}^\sharp \subseteq \Phi_{\dR}$. Let $\lambda_i,\ldots, \lambda_k \in R^+$ such that $\widetilde \theta = \sum \lambda_i \widetilde\sigma_i$. By Lemma~\ref{lemma: ord loc im}, for some $s \in \N$ there are $\mu_1, \ldots, \mu_t \in R^+$ such that $\Ha^s\theta = \sum \mu_i \widetilde \theta_i$. In particular we must have
\[
\Ha^s \lambda_i \equiv 0 \bmod{p^w}.
\]
Since the ordinary locus is dense by Assumption~\ref{ass: ord dense}, we have that $\Ha$ is not a zero divisor in $R^+/p^w$, so $\lambda_i \equiv 0 \bmod{p^w}$ and $\theta = \sum \lambda_i \theta_i$ as required.
\end{proof}
\begin{prop} \label{prop: sharp works}
The $R^+$-module $\Theta_{\dR}^\sharp$ is free of rank $t$ and it is a direct summand of $(\Sigma_{\dR}^\sharp)^\otimes$. Moreover, $\{\theta_i\}$ is a basis of $\Theta_{\dR}^\sharp/p^w$.
\end{prop}
\begin{proof}
Arguing exactly as in the proof of Lemma~\ref{lemma: base Phi}, we have that $\Theta^\sharp_{\dR}$ is free with basis $\{\widetilde \theta_i\}$. Since $\{\widetilde \theta_i\} \subseteq \{\widetilde\sigma_i\}$ and the latter is a basis of $\Phi_{\dR}$, we have that $\Theta_{\dR}^\sharp$ is a direct summand in $\Phi_{\dR}$ and hence in $(\Sigma_{\dR}^\sharp)^\otimes$.
\end{proof}
From now, we consider $\Theta^\sharp_{\dR}$ as a locally free sheaf of $\mathcal{O}^+_{\Ig_n(v)}$-modules, and similarly for and $\Theta^+_{\dR}$, $\Sigma^+_{\dR}$ etc.
\begin{rmk} \label{rmk: caraiani scholze}
Let us work at infinite level and over the ordinary locus, i.e. over $\Ig_{\infty}(0)$. Let $\C_p$ be the completion of $\overline{\Q}_p$. Let $T= \T_p(A)$ be the Tate module of the universal abelian variety over $\Ig_{\infty}(0)$. By \cite{car-scho}, we have that $T \otimes_{\Z_p} \C_p$ is equipped with the \emph{Hodge-Tate filtration} $T_0 \subseteq T \otimes_{\Z_p} \C_p$. The level-infinity canonical subgroup $H_\infty$ gives an integral structure of $T_0$. The above proposition can be seen as the analogue of the main result of \cite[Section 2]{car-scho}, that says that the Hodge-Tate period map from the infinite level Shimura variety factors through the flag variety associated to $G$. Our result is slightly more precise since it works integrally and at finite level, but it is less general since the canonical subgroup exists only on $\Ig_n(v)$, while the result of \cite{car-scho} holds over the whole Shimura variety.
\end{rmk}
\begin{prop} \label{prop: indep emb}
Our construction does not depend on either the symplectic embedding $G \hookrightarrow \widetilde G$ or the choice of $W \subseteq V_{\Q_p}^\otimes$.
\end{prop}
\begin{proof}
This is proved as \cite[Lemmas 2.3.4 and 2.3.7]{car-scho}.
\end{proof}
\subsection{Iwahoric subgroups and the weight space} We assume that our Borel subgroup $B \subseteq M^\circ$ extends to a Borel $\mathcal{B} \subseteq \mathcal{M}^\circ$. Let $\mathcal{U}$ be the unipotent part of $\mathcal{B}$. We write $\mathcal{I} \subseteq \mathcal{M}^\circ(\Z_p)$ for the Iwahori subgroup of elements whose reduction modulo $p$ lies in $\mathcal{B}(\F_p)$. We denote by $\mathcal{B}^{\op}$ the opposite Borel of $\mathcal{B}$, and similarly for $\mathcal{U}$ etc. Let $\mathcal{N}^{\op} \subseteq \mathcal{U}^{\op}(\Z_p)$ given by elements whose reduction modulo $p$ is the identity and similarly for $\widetilde{\mathcal{N}}$. We then have the Iwahori decomposition
\[
\mathcal{B}(\Z_p) \times \mathcal{N}^{\op} \stackrel{\sim}{\longrightarrow} \mathcal{I}.
\]
From now on we switch to the adic setting for our groups, so for example $M^\circ$ will be the adic space over $\Spa(K, \mathcal{O}_K)$ associated to $M^\circ$ and $\mathcal{M}^\circ$. For any rational $w > 0$, we define $I_w$ as the adic subgroup of $M^\circ$ of integral elements whose reduction modulo $p^w$ lies in $I \bmod{p^w}$. We let $B_w \colonequals B \cap I_w$ and $T_w \colonequals T \cap I_w$.

We now define our weight space. We write $r$ for the rank of $T$. Our weight space is the adic space $\mc W$ over $\Spa(K,\mathcal{O}_K)$ associated to the completed group algebra $\mathcal{O}_K \llbracket T(\Z_p) \rrbracket \cong \mathcal{O}_K \llbracket (\Z_p^\ast)^r \rrbracket$. It satisfies
\[
\mc W((R,R^+))= \Hom_{\cont}(T(\Z_p), R^\ast)
\]
for any complete Huber pair $(R,R^+)$. We have that $\mathcal{W}$ is a finite union of open polydiscs. Let $w > 0$ be a rational number. We say that $\kappa \in \mc W ((R,R^+))$ is \emph{$w$-locally analytic} if $\kappa$ extends to a pairing $T_w \times \Spa(R,R^+) \to \mathbf{G}_m$. Any $\kappa \in \mc W((R,R^+))$ is $w$-locally analytic for some $w$. More generally, let $\mc U \subset \mc W$ be an open subset and let $\kappa_{\mc U}^{\un}$ be its universal character. We say that $\kappa_{\mc U}^{\un}$ is $w$-locally analytic if it extends to a pairing $T_w \times \mc U \to \mathbf{G}_m$ If $\mc U$ is quasi compact then it is $w$-locally analytic for some $w$.

Finally, note that $\mathcal{W}$ contains $X^\ast(T)$ as a Zariski dense subset. Moreover, the map $\chi \mapsto \chi'$ extends to an involution of $\mc W$ denoted in the same way. If $\kappa$ is $w$-locally analytic then the same is true for $\kappa'$.
\subsection{\texorpdfstring{Analytic induction and $p$-adic modular forms}{Analytic induction and p-adic modular forms}}
Let $\mathcal U \subseteq \mathcal W$ be an open subset with associated universal character $\kappa_{\mathcal{U}}^{\un}$ and assume that $\kappa_{\mathcal{U}}^{\un}$ is $w$-locally analytic. We set
\begin{gather*}
V_{\kappa_{\mathcal{U}}^{\un}}^{w-\an} \colonequals \{f \colon I_w \times \mathcal{U} \to \mathbf{A}^1 \mbox{ such that } f(ib)=\kappa_{\mathcal{U}}^{\un}(b)f(i)
\mbox{ for all } (i,b) \in I_w \times B_w\}.
\end{gather*}
We have that $V_{\kappa_{\mathcal{U}}^{\un}}^{w-\an}$ is a representation of $I_w$ via $(i \star f)(x) = f(i^{-1}x)$. If $\kappa \in X^\ast(T)^+$, we have an inclusion
\[
V_\kappa \hookrightarrow V_\kappa^{w-\an}
\]
that respects the action of $\mathcal{I}$. Note that $V_\kappa$ is finite dimensional while $V_\kappa^{w-\an}$ is usually an infinite dimensional Banach space.

We now introduce a relative version of the space $V_\kappa^{w-\an}$ over $\Ig_n(v)$. We work locally as in Subsection~\ref{subsec: modifications}, so $R$ is such that over $\Spa(R,R^+) \subseteq \Ig_n(v)$ we have that $\Sigma^\sharp_{\dR}$ and $\Theta^\sharp_{\dR}$ are free $R^+$-modules of rank $g$ and $t$. In particular, we have the universal basis $\{f_i\}_{i=1}^t$ of $\Theta^+_{\et,n}$. By Proposition~\ref{prop: sharp works}, we have the basis $\{\HT_w(f_i)\}_{i=1}^t$ of $\Theta^\sharp_{\dR}/p^w$ (the marked sections). It follows that we have a filtration of $\Theta^+_{\et,n}$
\[
\Fil_\bullet \Theta^+_{\et,n} = \{ 0 \subseteq \langle f_1\rangle \subseteq \langle f_1, f_2\rangle \cdots \subseteq \langle f_1, \ldots, f_t \rangle =  \Theta^+_{\et,n} \}
\]
and similarly for $\Fil_\bullet \Theta^\sharp_{\dR}/p^w$. We assume that the Borel $B \subseteq M^\circ$ stabilizes the filtration $\Fil_\bullet \Theta^+_{\et,n}$ (since all Borel are conjugated, this assumptions is  harmless). Let $\imath \colon \Sigma^+ \otimes R^+ \to \Sigma_{\dR}^\sharp$ be an isomorphism such that $\imath(\Theta^+ \otimes R^+) \subseteq \Theta^\sharp_{\dR}$. Recall that by definition of $\Ig_n(v)$ we have a canonical identification $\Theta^+/p^n = \Theta^+_{\et,n}$. We say that \emph{$\imath$ is Iwahoric} if
\begin{itemize}
 \item $\imath(\Fil_\bullet \Theta^+_{\et,n} \otimes R^+ \bmod{p^w} ) = \Fil_\bullet \Theta_{\dR}^\sharp/p^w$.
 \item $\imath(f_i \otimes 1 \bmod{p^w}) \equiv \HT_w(f_i) \bmod{\Fil_{i-1} \Theta_{\dR}^\sharp/p^w }$ for all $i=1,\ldots,t$.
\end{itemize}
We define a functor $\Iw_w \to \Ig_n(v)$ following the construction of \cite[Subsection~2.4]{AIPicm}. Its value on $\Spa(R,R^+)$ is 
\begin{gather*}
\Iw_w(R,R^+) \colonequals \{ \imath \in \underline{\mathrm{Iso}}(\Sigma^+ \otimes \mathcal{O}^+_{X^{\ad}}, \Sigma_{\dR}^\sharp) \mbox{ such that } \\
\imath(\Theta^+ \otimes \mathcal{O}^+_{X^{\ad}}) \subseteq \Theta^\sharp_{\dR} \mbox{ and } \imath \mbox{ is Iwahoric} \}.
\end{gather*}
One can check that $\Iw_w$ is representable. We define a left action $I_w \times \Iw_w \to \Iw_w$ by sending an isomorphism $\imath$ to $\imath \circ g^{-1}$ for all $g \in I_w$. By definition $f \colon \Iw_w \to \Ig_n(v)$ is an $I_w$-torsor.
\begin{defi}
Let $\mathcal{U} \subseteq \mathcal{W}$ be an open subset with associated universal character $\kappa_{\mathcal{U}}^{\un}$ and assume that $\kappa_{\mathcal{U}}^{\un}$ is $w$-locally analytic. We define the sheaf $\omega^{\kappa_{\mathcal{U}}^{\un}}_w$ on $\Ig_n(v) \times \mathcal{U}$ by
\[
\omega^{\kappa_{\mathcal{U}}^{\un}}_w \colonequals (f \times 1)_\ast \mathcal{O}_{\Iw_w \times \mathcal{U}}[\kappa_{\mathcal{U}}^{\un}],
\]
meaning the sheaf of functions that are $\kappa_{\mathcal{U}}^{\un}$-equivariant for the action of $B_w$. If $\kappa \in \mathcal{U}$ is a character we similarly have the sheaf $\omega^{\kappa}_w$ on $\Ig_n(v)$. By construction, the pullback of $\omega^{\kappa_{\mathcal{U}}^{\un}}_w$ over $\Ig_n(v) \times \{\kappa\}$ is $\omega^{\kappa}_w$.

We call an element of
\[
\M_w^{\kappa_{\mathcal{U}}^{\un}} \colonequals \mathrm{H}^0_b(\Ig_n(v) \times \mathcal{U},\omega^{\kappa_{\mathcal{U}}^{\un}}_w),
\]
where by $\mathrm{H}^0_b$ we mean bounded sections, a \emph{$p$-adic family of $w$-locally analytic $v$-overconvergent modular forms of weight $\kappa_{\mathcal{U}}^{\un}$}. If $\kappa \in \mathcal{W}$, we define the space $\M_w^\kappa$ similarly. Note that in this case we have the \emph{specialization morphism} $\M_w^{\kappa_{\mathcal{U}}^{\un}} \to \M_w^\kappa$.
\end{defi}
\section{Hecke operators} \label{sec: Hecke}
We now introduce the Hecke algebra and its action on the space of families of modular forms. We fix for all this section a $w$-analytic character $\kappa \in \mathcal{W}$. Everything can be extended without problems to the case of an affinoid $\mathcal{U} \subseteq \mathcal{W}$ with associated universal character $\kappa_{\mathcal{U}}^{\un}$ and we assume that $\kappa_{\mathcal{U}}^{\un}$ is $w$-locally analytic (and the specialization map will be equivariant for the Hecke action).
\subsection{\texorpdfstring{Hecke operators outside $p$}{Hecke operators outside p}} We start by defining the action of the unramified Hecke algebra. Let $\ell \neq p$ be a prime such that $G$ is unramified at $\ell$ and let us fix $K_\ell \subseteq G(\Q_\ell)$ an hyperspecial subgroup. We denote by $\mathcal{H}_\ell$ the unramified Hecke algebra of $G$ at $\ell$ with $\Z_p$ coefficients (using the Haar measure such that $K_\ell$ has volume $1$). Our unramified Hecke algebra is
\[
\mathcal{H} \colonequals \bigotimes_{\ell}{'} \mathcal{H_\ell},
\]
where by $\bigotimes{'}$ we mean restricted tensor product. To define an action of $\mathcal{H}$ it is enough to define an operator $\T_g$ for all $g \in G(\Q_\ell)$ (depending as usual only on the double class of $g$). Recall our fixed compact open subgroup $K \subseteq G(\A_f)$. For any $g \in \Q_\ell$, we set $K_g \colonequals K \cap gKg^{-1}$, another compact open subgroup of $G(\A_f)$. We can repeat all the above constructions for $K_g$ instead of $K$, obtaining the adic space $\Ig_n^g(v)$. We have two finite étale morphisms
\[
p_1, p_2 \colon \Ig_n^g(v) \to \Ig_n(v).
\]
The first one is the natural projection induced by $K_g \subseteq K$ and the second one is the natural projection $\Ig_n^{g^{-1}}(v) \to \Ig_n(v)$ composed with the isomorphism $\Ig_n^g(v) \cong \Ig_n^{g^{-1}}(v)$ given by the action of $G(\A_f)$ on the tower of Shimura varieties. The pullback to $\Ig_n^g(v)$ of the universal abelian variety over $\Ig_n(v)$ via $p_1$ (resp. $p_2$) will be denoted by $A_1$ (resp $A_2$). It follows that, over $\Ig_n^g(v)$ we have a universal isogeny $\pi \colon A_1 \to A_2$. We get a morphism between the dual of the canonical subgroups $\pi^{\D} \colon H_n^{\D}(A_2) \to H_n^{\D}(A_1)$. We have that $\pi^{\D}$ is an isomorphism, so, as in \cite[Lemma~6.1.1]{AIP}, the induced map $\pi^\ast \colon p_2^\ast \Sigma^\sharp_{\dR} \to p_1^\ast \Sigma^\sharp_{\dR}$ is an isomorphism.
\begin{lemma} \label{lemma: hecke not p}
The isomorphism $\pi^\ast \colon p_2^\ast \Sigma^\sharp_{\dR} \stackrel{\sim}{\longrightarrow} p_1^\ast \Sigma^\sharp_{\dR}$ induces $\pi^\ast \colon p_2^\ast \Theta^\sharp_{\dR} \stackrel{\sim}{\longrightarrow} p_1^\ast \Theta^\sharp_{\dR}$ and in particular we get an isomorphism
\[
\pi^\ast \colon p_2^\ast \Iw_w \stackrel{\sim}{\longrightarrow} p_1^\ast \Iw_w.
\]
\end{lemma}
\begin{proof}
We have that $\pi^{\D} \colon H_n^{\D}(A_2) \stackrel{\sim}{\longrightarrow} H_n^{\D}(A_1)$ induces $\pi^{\D} \colon p_2^\ast \Theta^+_{\et,n} \stackrel{\sim}{\longrightarrow} p_1^\ast \Theta^+_{\et,n}$. The commutative diagram
\[
\begin{tikzcd}
p_2^\ast \Theta^\sharp_{\dR} \arrow{r}{\pi^\ast} \ar{d} & p_1^\ast \Theta^\sharp_{\dR} \ar{d} \\
p_2^\ast \Theta^\sharp_{\dR}/p^w \arrow{r} & p_1^\ast \Theta^\sharp_{\dR}/p^w \\
p_2^\ast \Theta^+_{\et,n} \arrow{r}{\sim} \ar{u}{\HT_w} & p_1^\ast \Theta^+_{\et,n} \ar{u}[swap]{\HT_w}
\end{tikzcd}
\]
shows, since the linearization of $\HT_w$ is an isomorphism, that $p_2^\ast \Theta^\sharp_{\dR} \to p_1^\ast \Theta^\sharp_{\dR}$ is surjective. But $p_2^\ast \Theta^\sharp_{\dR}$ and $p_1^\ast \Theta^\sharp_{\dR}$ are locally free modules of the same rank, so $\pi^\ast$ is an isomorphism. The lemma follows.
\end{proof}
Using the trace $\tr(p_1)$ of $p_1$, we obtain a morphism
\begin{gather*}
\widetilde{\T}_g \colon \Homol^0(\Ig_n(v), \mc O_{\Iw_w}) \stackrel{p_2^{\ast}}{\longrightarrow} \Homol^0(\Ig_n^g(v), p_2^\ast \mc O_{\Iw_w}) \stackrel{\pi^{\ast -1}}{\longrightarrow}
\Homol^0(\Ig_n^g(v), p_1^\ast \mc O_{\Iw_w}) \stackrel{\tr(p_1)}{\longrightarrow} \\
\stackrel{\tr(p_1)}{\longrightarrow} \Homol^0(\Ig_n(v), \mc O_{\Iw_w})
\end{gather*}
We have that $\widetilde{\T}_g$ preserves $\kappa$-homogeneous and bounded sections, so we get the required Hecke operator
\[
\T_g \colon \M_w^{\kappa} \to \M_w^{\kappa}
\]
\subsection{\texorpdfstring{Hecke operators at $p$}{Hecke operators at p}} We now define a completely continuous operator $\U$ on the space of overconvergent modular forms.

Recall that $\widetilde S^{\ad}$ is the Siegel variety. We define $\widetilde S(v)^{\ad}$ in the obvious way using the Hasse invariant and we consider the morphism $\Frob \colon \widetilde S(v)^{\ad} \to \widetilde S(pv)^{\ad}$ given by taking the quotient by the first canonical subgroup (see \cite[Théorème~8, (2)]{fargues_can} for the behavior of the Hasse under the quotient by the canonical subgroup). It is finite and flat.
\begin{prop} \label{prop: frob resp Hodg}
The morphism $\Frob$ factors through $S(v)^{\ad}$, giving a finite and flat morphism
\[
\Frob \colon S(v)^{\ad} \to S(pv)^{\ad}.
\]
\end{prop}
\begin{proof}
It is enough to prove that, for all $x \in S(v)^{\ad}$, we have $\Frob(x) \in S(pv)^{\ad}$. We first of all prove the claim for $v = 0$, i.e. over the ordinary locus, using Serre-Tate theory. If $x$ lies in the ordinary locus, then the completed local ring of $\widetilde S(0)^{\ad}$ at $x$ is a formal torus and $\Frob$ is a group morphism by \cite[Lemma~4.1.2]{serre-tate}. By the main result of \cite{noot}, the closed immersion $S(0)^{\ad} \hookrightarrow \widetilde S(0)^{\ad}$ is given by a formal subtorus, so it must be preserved by $\Frob$.

The situation is locally the following: we have an affinoid $\Spa(\widetilde R(v), \widetilde R(v)^+)$ inside $\widetilde S(v)^{\ad}$, a closed subset $\Spa(R(v),R(v)^+)$ given by an ideal $I(v) \subseteq \widetilde R(v)^+$ and similarly for $\Spa(\widetilde R(pv), \widetilde R(pv)^+)$ etc. We also have a morphism $\varphi \colon \widetilde R(pv)^+ \to R(v)^+$. We need to prove that $\varphi(I(pv)) \subseteq I(v)$. We know that there is an element $\Ha \in \widetilde R(v)^+$, corresponding to the Hasse invariant, with the property that $\varphi(I(pv)) \subseteq \widetilde R(v)^+[\Ha^{-1}] I(v)$. In particular, for each $i \in I(pv)$, there is $k \in \N$ such that $\Ha^k \varphi(i) \in I(v)$. Since $\Ha$ is not a zero divisor in $R(v)^+ = \widetilde R(v)^+/I(v)$, we obtain that $\varphi(i) \in I(v)$ as required.
\end{proof}
We write again $\Frob \colon \Ig_n(v) \to \Ig_n(pv)$ for the induced morphism. Let $A(v) \to \Ig_n(v)$ be the universal abelian scheme with canonical subgroup $H_1(A(v))$, and similarly for $A(pv)$. By construction we have $\Frob^\ast \omega_{A(pv)} = \omega_{A(v)'}$, where we set $A(v)' \colonequals A(v)/H_1(A(v))$. We consider the universal isogeny $\pi \colon A(v) \to A'(v)$ and its dual $\pi^{\D} \colon A'(v) \to A(v)$. Since the kernel of $\pi^{\D}$ intersects trivially the canonical subgroup of $A'$ we have, as in the proof of Lemma \ref{lemma: hecke not p}, that $\pi^{\D}$ induces isomorphisms, denoted $\pi^{\D \ast}$,
\[
\Sigma^\sharp_{\dR} \stackrel{\sim}{\longrightarrow} \Sigma^{\sharp \prime}_{\dR}, \; \Theta^\sharp_{\dR} \stackrel{\sim}{\longrightarrow} \Theta^{\sharp \prime}_{\dR} \text{ and } \Iw_w \stackrel{\sim}{\longrightarrow} \Iw_w',
\]
where the $'$ means that the object is defined using $A'$ rather then $A$. Note that $\Iw_w' = \Frob^\ast \Iw_w$, where the latter $\Iw_w$ lives over $\Ig_n(pv)$. We obtain a morphism
\[
\widetilde{\U} \colon \Homol^0(\Ig_n(v), \mc O_{\Iw_w}) \stackrel{\pi^{\D \ast -1}}{\longrightarrow}
\Homol^0(\Ig_n(v), \Frob^\ast \mc O_{\Iw_w}) \stackrel{\tr(\Frob)}{\longrightarrow} \Homol^0(\Ig_n(pv), \mc O_{\Iw_w})
\]
We have that $\widetilde{\U}$ preserves $\kappa$-homogeneous and bounded sections. Taking the composition with the restriction from $\Ig_n(pv)$ to $\Ig_n(v)$, we finally obtain the Hecke operator
\[
\U \colon \M_w^{\kappa} \to \M_w^{\kappa}.
\]
Since $\U$ improves the radius of overconvergence, it is completely continuous.
\begin{rmk} \label{rmk: no norm}
Usually, in the definition of the $\U$ operator, one uses a factor of normalization to optimize its action on the space of integral modular forms. Since we only work with modular forms defined in characteristic $0$, the absence of the normalization factor is harmless for us.
\end{rmk}
\begin{defi} \label{defi: Hecke alg}
We will write $\mathbf{T} \colonequals \mathcal{H} \otimes \Z[\U]$ for the Hecke algebra generated by all the operators defined above.
\end{defi}
\subsection{The eigenvariety in the proper case} \label{subsec: eigenvar proper} As explained in the Introduction, we do not treat in this work any compactification and in particular we do not construct the eigenvariety in general. We suppose in this subsection that our Shimura variety $\mathcal{S}$ is \emph{proper}. It follows that the ordinary locus of the reduction of $\mathcal{S}$ is affine and in particular $S(v)^{\ad}$ and $\Ig_n(v)$ are affinoid.

Fix an affinoid $\mathcal{U} = \Spm(A) \subseteq \mathcal{W}$, with associated $w$-analytic universal character $\kappa_{\mathcal{U}}^{\un}$. The sheaf $\omega_w^{\kappa_{\mathcal{U}}^{\un}}$ is a Banach sheaf on $\Ig_n(v) \times \mathcal{U}$ (where $v$ depends on $\mathcal{U}$ as usual) in the sense of \cite[Appendix A]{AIP}. Recall the notion of \emph{projective Banach} $A$-module introduced by Buzzard in \cite{buzz_eigen}. (See also \cite[Subsection 8.1.2]{AIP}.)
\begin{prop} \label{prop: Pr proper case}
Let $\mathcal{U} = \Spm(A)$ and $\kappa_{\mathcal{U}}^{\un}$ be as above and let $\kappa \in \mathcal{U}$ be a character.
\begin{itemize}
 \item We have that $\M_w^{\kappa_{\mathcal{U}}^{\un}}$ is a projective Banach $A$-module.
 \item The specialization morphism $\M_w^{\kappa_{\mathcal{U}}^{\un}} \to \M_w^{\kappa}$
is surjective.
\end{itemize}
\end{prop}
\begin{proof}
Since $\Ig_n(v)$ is affinoid, this is proved exactly as \cite[Proposition 8.2.3.3]{AIP}.
\end{proof} 
Keeping the above notations, let $\mathcal{Z}$ be  the spectral variety associated to the characteristic power series $P(T)$ of the completely continuous operator $\U$ acting on the space $\M_w^{\kappa_{\mathcal{U}}^{\un}}$. It is the closed subspace of $\Spm(A) \times \mathbb{A}^1$ defined by $P(T)=0$. Buzzard's construction in \cite{buzz_eigen} gives immediately the following result.
\begin{prop} \label{prop: eigen proper}
There exist a rigid analytic space $\pr:\mc E \rightarrow \mathcal{Z}$, called the \emph{eigenvariety} which satisfies the following properties:
\begin{itemize}
\item $\mc E$ is equidimensional of dimension $\dim(A)$.
\item Let $\eta$ be the structural morphism $\mathcal{E} \rightarrow \mathrm{Spm}(A)$. The space $\mathcal{E}$ parameterizes finite slope systems of eigenvalues appearing in $\M_w^{\kappa_{\mathcal{U}}^{\un}}$; more precisely, for every $x \in \Spm(A)$, each point in $\eta^{-1}(x)$ corresponds to a system of eigenvalues for $\mathbf{T}$ inside $\M_w^{\kappa_{\mathcal{U}}^{\un}}\otimes_A \overline{\kappa(x)}$ which is of finite slope for $\U$.
\item The map $\mr{pr}$ is finite. The map $\eta$ is locally finite.
\item The module $\M_w^{\kappa_{\mathcal{U}}^{\un}}$ defines a coherent sheaf $\mathcal{M}$. Let $(x, \lambda) \in \mathcal{Z} \subseteq \Spm(A) \times \mathbb{A}^1$ be a point of the spectral variety. The fiber $\mathcal{M}_{(x,\lambda)}$ is the generalized eigenspace for $\lambda^{-1}$ inside $\M_w^{\kappa_{\mathcal{U}}^{\un}}\otimes_A \overline{\kappa(x)}$.
\end{itemize}
\end{prop}

\section{An example: orthogonal Shimura varieties} \label{sec: orthogonal}
In this section we want to make slightly more explicit our general construction for a (non PEL) Shimura variety of Hodge type: the $\Spin$ variety. We follow the presentation of \cite{madapusi}

Let $L$ be a $\Q$-vector space endowed  with a quadratic form $Q \colon L \to \Q$. We assume that $Q$ is non-degenerate of signature $(n,2)$. We denote by $C = C(L)$ the \emph{Clifford algebra} of $L$. It is the quotient of the tensor algebra $\bigotimes L$ by the subspace generated by the elements of the form $v \otimes v - Q(v) \cdot 1$. It is equipped with a natural inclusion $L \hookrightarrow C$ and it is universal among the $\Q$-algebras $R$ equipped with a morphism $f \colon L \to R$ such that $f(v)^2 = Q(v)$. The natural grading on $\bigotimes L$ induces a $\Z/2$-grading on $C$,
\[
C = C^+ \oplus C^-
\]
such that $C^+$ is a subalgebra of $C$. There is a unique anti-involution $^\star \colon C \to C$ that restricts to the identity on $L$. We fix an element $\delta \in C^\ast$ such that $\delta^\star = - \delta$. Let $\Trd \colon C \to \Q$ be the \emph{reduced trace map}. We define a bilinear map
\begin{gather*}
\Psi \colon C \times C \to \Q \\
(x,y) \mapsto \Trd(x\delta y^\star)
\end{gather*}
One can check that $\Psi$ is a symplectic form on $C$. We now define $G = \GSpin(L)$, a reductive group over $\Q$, by the formula
\[
G(R) \colonequals \GSpin(L)(R) \colonequals \{ x \in (C^+_R)^\ast \text{ such that } x (L_R) x^{-1} = L_R \}
\]
for any $\Q$-algebra $R$. We let $\GSpin(L)$ act on $C$ be left multiplication. We obtain in this way an inclusion of algebraic groups
\[
\GSpin(L) \hookrightarrow \GSp(C),
\]
where $\GSp(C)$ is the symplectic group associated to $\Psi$.

We now let $X$ be the space of oriented negative definite $2$-planes in $L_{\R}$. As explained in \cite[Subsection~3.1]{madapusi}, if $\delta$ satisfies certain technical conditions, we have that $X$ is a complex manifold that can be identified with a $G(\R)$-conjugacy class in $\Hom(\Res_{\C/\R} \mathbf{G}_{\mathrm{m}}, G_{\R})$. In particular, the pair $(G,X)$ is a Shimura datum. The reflex field is $\Q$ and, thanks to the embedding $G \hookrightarrow \GSp(C)$, we have that $(G,X)$ is of Hodge type (see \cite[Subsection~3.5]{madapusi} for how to embed $X$ in a union of Siegel upper half-spaces).

We assume that there is a maximal $\Z$-lattice $\Lambda \subseteq L$ that is self-dual at $p$ under the pairing induced by $Q$. We are then in the situation of Subsection~\ref{subsec: hodge type} and, after choosing a suitable compact open level $K = K^p \times K_p \subseteq G(\A_f)$, we get the Shimura variety $S^{\ad} \to \Spa(\Q_p,\Z_p)$. Fixing a Borel subgroup and a maximal torus $T$ of $G$ we can apply our general theory and we obtain, for any character $\kappa$ of $T$, the sheaf $\omega^{\kappa}$ on $\Ig_n(v)$. (Here $v$ is small enough, $\kappa$ is $w$-locally analytic and $n$ depends on $v$ and on $w$.) To define $\omega^{\kappa}$ it is essential to consider $\Theta^+$, $\Theta^+_{\et,n}$ etc. In the case of $\GSpin$-Shimura varieties, these objects can be made more explicit, as we are going to explain.

We write $H$ for the Clifford algebra $C$ viewed as a $G$-representation. It is associated with the universal abelian scheme $A \to \Ig_n(n)$, in the sense that the associated automorphic sheaf $\mathcal{E}(H)$ is the first de Rham cohomology $\Homol^1_{\dR}(A)$ of $A$. We can now explicitly describe the subspace $W \subseteq H^\otimes$ that characterize $G$. We let $G$ acts on $L$ by $g \cdot \ell \colonequals g\ell g^{-1}$. Note that the natural inclusion $L \hookrightarrow H$ \emph{is not} $G$-equivariant (since $G$ acts on $C$ via left multiplication). On the other hand, there is a natural action, by left multiplication, of $L$ on $C$ and in particular there is a $G$-equivariant embedding
\[
L \hookrightarrow \End_C(H) = H \otimes H^\vee.
\]
This gives the subspace $W \subseteq H^\otimes$, indeed one can prove that
\[
G = \{ g \in \GSp(C) \text{ such that } g(L) = L \}.
\]
We see that in this case $W$ is canonical. Attached to $L$ there is a variation of Hodge structure of type $(1,-1)$, $(-1,1)$, and $(0,0)$. The graded object of the automorphic sheaf associated to $L$ is the sheaf called $\Theta_{\dR}$ above. We have that $\Theta_{\dR}$ is a direct summand of $\omega \otimes \omega^\vee$, where $\omega \colonequals \omega_A$. We then have $\Theta_{\dR}^\sharp =\bigl( \omega^\sharp \oplus (\omega^\sharp)^\vee\bigr) \cap \Theta_{\dR}$. Thanks to the lattice $\Lambda \subseteq L$, we have a $\Z$-lattice $H^+ \subseteq H$. Over $\Ig_n(v)$, there is a canonical identification $H_1^{\D} = H^+/p^n$ that gives a direct summand $\Theta^+_{\et,n} \cong \Lambda/p \subseteq (H^+/p^n) \oplus (H^+/p^n)^\vee$. Proposition \ref{prop: sharp works} says that the image via the Hodge-Tate morphism of $\Theta^+_{\et,w}$ is inside $\Theta^\sharp_{\dR}/p^w$ and that the linearization of $\HT_w$ is an isomorphism. In particular $\Theta^\sharp_{\dR}$ is locally free. The space $\Iw_w$ parameterizes filtration of $\omega^\sharp$ such that the induced filtration on $\Theta^\sharp_{\dR}$ agrees, via $\HT_w$, with the fixed filtration on $\Lambda/p^w$ given by the choice of a Borel subgroup of $G$ (and similarly for a basis of the graded pieces).

One can also use the notion of \emph{tensors} rather then our subspace $W$: this is for example the point of view taken in \cite{kisin} and \cite{car-scho}. There exist an idempotent morphism $\pi \in \End(\End(H))= H^{\otimes 2} \otimes H^{\vee \otimes 2}$ such that $G \subseteq \GSp(C)$ is the stabilizer of $\pi$, see \cite[Lemma~1.4]{madapusi}. The image of $\pi$ is $L \subseteq \End(H)$. We have the de Rham realization $\pi_{\dR} \in \End(\End(\omega))$. Its image is $\Theta_{\dR}$. All these objects admit canonical integral structure, but we think that to work with modifications it is more convenient to consider directly $\Theta_{\dR}$ rather then $\pi_{\dR}$.
\subsection{Families of orthogonal modular forms and the eigenvariety} \label{subsec: orth form} We explain in this subsection the construction of the eigenvariety for the modular forms considered in Borcherds' theory as in \cite{howard-pera}. Since the situation is simpler than the general case, we only sketch the arguments.

Recall that by definition $\Theta = \Gr^\bullet(W)$. In the theory of Borcherds products one is interested, rather than in the graded object, to $\Xi \colonequals \Fil^1(W)$ that is a line giving the part of type $(1,-1)$ of the above Hodge structure. See for example \cite[Section 4.2]{howard-pera}. We adapt all our notations and construction to $\Xi$ in the obvious way. In particular we have $\Xi_{\dR}$, that is a line bundle. A simple variation of our construction allows to interpolate the sheaves $\Xi_{\dR}^{\otimes k}$, where $k \in \Z$, defining the sheaf $\Xi_{\dR}^{\kappa}$ for any $p$-adic integer $\kappa \in \Z_p$. Indeed, since we are interpolating line bundles, the construction is much easier and $\Xi_{\dR}^\kappa$ is a line bundle for any $\kappa$. (This is the case, for example, of \cite{over}.) In this situation the relevant weight space $\mathcal{W}_{\Xi}$ is a $1$-dimensional closed subspace of the weight space for $G$. If $\mathcal{U} = \Spm(A) \subseteq \mathcal{W}_{\Xi}$ is an affinoid with associated universal character $\kappa_{\mathcal{U}}^{\un}$, we have the sheaf $\Xi_{\dR}^{\kappa_{\mathcal{U}}^{\un}}$ on $\Ig_n(v) \times \mathcal{U}$. Contrary to the general situation, this sheaf is coherent and there is no notion of local analyticity.

By \cite{howard-pera}, the sheaf $\Xi^+_{\dR}$ extends to a fixed toroidal compactification $\Ig_n(v)^{\tor}$ of $\Ig_n(v)$, and so do the sheaf $\Xi_{\dR}^{\kappa}$ for all $\kappa \in \Z_p$. The same is true for $\Xi_{\dR}^{\kappa_{\mathcal{U}}^{\un}}$. The space of \emph{families of orthogonal modular forms} is defined as
\[
\M_{\Xi}^{\kappa_{\mathcal{U}}^{\un}} \colonequals \mathrm{H}^0(\Ig_n(v)^{\tor} \times \mathcal{U},\Xi_{\dR}^{\kappa_{\mathcal{U}}^{\un}}).
\]
Similarly we have the space $\M_{\Xi}^{\kappa}$ of modular forms of weight $\kappa$, for any $\kappa \in \Z_p$ and the specialization morphism $\M_{\Xi}^{\kappa_{\mathcal{U}}^{\un}} \to \M_{\Xi}^{\kappa}$ if $\kappa \in \mathcal{U}$. There is an action of the Hecke algebra $\mathbf{T}$ on $\M_{\Xi}^{\kappa_{\mathcal{U}}^{\un}}$ and on $\M_{\Xi}^{\kappa}$. The $\U$-operator is completely continuous. Let $\pi \colon \Ig_n(v)^{\tor} \times \mathcal{U} \to \Ig_n(v)^{\min} \times \mathcal{U}$ be the natural projection. (The minimal compactification is constructed for example in \cite{madapusi}.) We have that $\pi_\ast \Xi_{\dR}^{\kappa_{\mathcal{U}}^{\un}}$ is again a coherent sheaf and that $\Ig_n(v)^{\min}$ is affinoid, so Tate's acyclicity holds. This implies that $\M_{\Xi}^{\kappa_{\mathcal{U}}^{\un}}$ is a projective $A$-module and we can apply Buzzard's construction obtaining the following result.
\begin{prop} \label{prop: eigen orth}
There exist an eigenvariety $\mathcal{E} \to \mathcal{U}$ that parameterizes finite slope system of eigenvalues appearing in $\M_{\Xi}^{\kappa_{\mathcal{U}}^{\un}}$. Moreover, if $\kappa \in \mathcal{U}$ then the specialization morphism $\M_{\Xi}^{\kappa_{\mathcal{U}}^{\un}} \to \M_{\Xi}^{\kappa}$ is surjective.
\end{prop}

\bibliographystyle{amsalpha}
\bibliography{biblio}

\end{document}